\documentclass[12pt]{amsart}
\usepackage[textwidth=16cm, textheight=22cm]{geometry}
\usepackage{graphicx}
\usepackage{latexsym}

\newtheorem{proposition}{Proposition}

\def\RR{\overline{\mathbb{R}}}
\def\RRn{\RR^n}

\def\R{\mathbb{R}}

\def\Rp{\overline{\mathbb{R}}}
\def\Rpn{\Rp^n}
\def\Rpnn{\Rp^{n\times n}}

\def\supp{\operatorname{supp}}

\def\bzero{\mathbf{0}}
\def\bunity{\mathbf{1}}

\def\bez{\backslash}

\def\oI{\overline{I}}
\def\oW{\overline{W}}

\begin{document}
\title{Basic solutions of systems with two max-linear inequalities}
\thanks{This research was supported by EPSRC grant RRAH12809 and
RFBR grant 08-01-00601 (S. Sergeev) and 
NRC grant RGPIN-143068-05 (E. Wagneur)}

\author{Serge\u{\i} Sergeev}
\address{University of Birmingham,
School of Mathematics, Watson Building, Edgbaston B15 2TT, UK}
\email{sergeevs@maths.bham.ac.uk}

\author{Edouard Wagneur}
\address{GERAD and
D\'{e}partement de Math\'{e}matiques et G\'{e}nie Industriel
\'{E}cole Polytechnique de Montr\'{e}al
Montr\'{e}al, PQ, Canada}
\email{Edouard.Wagneur@gerad.ca}


\begin{abstract}
We give an explicit description of the
basic solutions of max-linear systems $A\otimes x\leq B\otimes x$
with two inequalities.
\end{abstract}
\maketitle

\section{Introduction}

We consider systems of two max-plus linear inequalities
\begin{equation}
\label{e:sysmain}
\begin{split}
a_{11}\otimes x_1\oplus\ldots\oplus 
a_{1n}\otimes x_n &\leq 
b_{11}\otimes x_1\oplus\ldots\oplus b_{1n}\otimes x_n,\\
a_{21}\otimes x_1\oplus\ldots\oplus a_{2n}\otimes x_n &\leq b_{21}\otimes x_1\oplus\ldots\oplus b_{2n}\otimes x_n.
\end{split}
\end{equation}
Here $\otimes:=+,$ $\oplus:=\max,$ and
$a_{ij},b_{ij},x_j\in\R\cup\{-\infty\}$ for $i=1,2$ and $j=1,\ldots,n$.

General systems of max-linear inequalities (equivalently, equalities)
were tackled by Butkovi\v{c} and Heged\"{u}s \cite{BH-84} who established
an elimination method for finding basic solutions of such systems,
starting with basic solutions of 
just one equation or inequality and adding all other constraints one by one.
This algorithm served as a proof that solution sets to 
max-linear systems have finite bases, and it did not
seem to be efficient enough for practical implementation.
But at present, Allamegeon, Gaubert and Goubault \cite{AGG-09} have come
up with a nouvelle approach to the scheme of \cite{BH-84}, in which
every step of adding new constraint is dramatically
improved by using a max-plus analogue of double description
method, and also a certain criterion of minimality estblished
in \cite{BSS-07,GK-07,NS-07II}, see also \cite{Ser-08}, which allows to efficiently test the extremality of a generator.

The idea of the present paper is that
when the number of inequalities is small,
the basic solutions can be written out explicitly.
However as shown by Wagneur, Truffet, Faye and Thiam 
\cite{WTFT}, even in the case of two inequalities \eqref{e:sysmain}
the number of generators is large and the problem
to establish a systematic classification and 
to resolve the extremality 
by writing out explicit conditions is nontrivial. This goal
is achieved in the present paper by 1) representing the set of
all solutions as the union of cones generated
by certain Kleene stars (Section
\ref{s:gather}), 2) selecting basic solutions by means
of the above mentioned criterion of minimality \cite{BSS-07,GK-07,NS-07II} (Section \ref{s:extr}) which
we call the multiorder principle \cite{Ser-08}.
This leads to explicit description of basic solutions
and to a procedure which finds all of them in no more than $O(n^3)$
operations.

\section{Gathering the generators}
\label{s:gather}

We work with the analogue of linear algebra
developed over the max-plus semiring $\R_{\max,+}$ which is
the set of real numbers with adjoined minus infinity $\RR=\R\cup\{-\infty\}$ equipped with the operations
of ``addition'' $a\oplus b:=\max(a,b)$ and 
``multiplication'' $a\otimes b:=a+b$. 
Zero $\bzero$ and unity $\bunity$ of this semiring are equal,
respectively, to $-\infty$ and $0$. 
The operations of the
semiring are extended to the nonnegative matrices and vectors in the
same way as in conventional linear algebra. That is if $A=(a_{ij})$,
$B=(b_{ij})$ and $C=(c_{ij})$ are matrices of compatible sizes with
entries from $\Rp$,
we write $C=A\oplus B$ if $c_{ij}=a_{ij}\oplus b_{ij}$ for all $i,j$ and $%
C=A\otimes B$ if $c_{ij}=\bigoplus_k a_{ik}\otimes
b_{kj}=\max_{k}(a_{ik}\otimes b_{kj})$ for all $i,j$. 
The notation $\otimes$ will be often omitted.

The main geometrical object of this max-plus linear algebra is a subset $K\in\RRn$ closed under the operations of componentwise maximization $\oplus$ and ``multiplication''
$\otimes$ by scalars (which means addition in the conventional sense). Such subsets are called
{\em max-plus cones} or just {\em cones} 
if there is no mix up with the ordinary convexity.

A vector $x\in\Rpn$ is a {\em (max-linear) combination} of $y_1,\ldots, y^m\in S$. A set $S\subseteq\Rpn$ is {\em generated} by 
$y^1,\ldots,y^m\in S$ if each $x\in S$ 
is a max-linear combination of $y^1,\ldots,y^m$. When
vectors arise as columns (resp. rows) of matrices, it will
be convenient to represent them as max-linear combinations
of the {\em column unit vectors}
\begin{equation}
\label{col-units}
e_i=(\overbrace{\bzero\ldots \bzero}^{i-1}\ \bunity\ \bzero\ldots \bzero)',
\end{equation}
respectively the {\em row unit vectors} $e'_i$, which are their 
transpose.

The following series is called the {\em Kleene star} of $A$:
\begin{equation}
\label{def-kls}
A^*=I\oplus A\oplus A^2\ldots,
\end{equation}
where $I$ is the max-plus unity matrix, which has all diagonal entries $\bunity$ and all off-diagonal entries $\bzero$. 
When $A^*$ has finite entries (in other words, converges)
it is easily shown that $A\otimes x\leq x$ is equivalent to 
$A^*\otimes x=x$. We also have the following.
\begin{proposition}
\label{klstar-gen}
Let $A\in\Rpn$ be such that $A^*$ has finite entries.
Then $\{x\mid A\otimes x\leq x\}$ is generated by the
columns of $A^*$.
\end{proposition}

This section will be based on the following two
observations. In the formulation we use the {\em row unit vectors} $e'_i$. We denote by $A_{i\cdot}$, resp. $A_{\cdot i}$,
the $i$th row, resp. the $i$th column, of $A$. 

\begin{proposition}
\label{verysimple}
Let $A\in\Rpnn$ have rows
\begin{equation}
\label{aicdot}
A_{i\cdot}=
\begin{cases}
e'_k\oplus \bigoplus_{l\neq k} a_{kl} e'_l, & \text{if $i=k$},\\
e'_i, \text{otherwise},
\end{cases}
\end{equation}
for $i=1,\ldots,n$. Then the set
$\{x\mid A\otimes x\leq x\}$ is generated by
the columns of $A$.
\end{proposition}
\begin{proof}
In this case $A^*=A$, after which Proposition \ref{klstar-gen} is applied. 
\end{proof}

\begin{proposition}
\label{alsosimple}
Let $A\in\Rpnn$ have rows
\begin{equation}
\label{aicdot2}
A_{i\cdot}=
\begin{cases}
e'_k\oplus \bigoplus_{l\in L_1} a_{kl} e'_l, & \text{if $i=k$},\\
e'_m\oplus \bigoplus_{l\in L_2} a_{ml} e'_l, & \text{if $i=m$},\\
e'_i, \text{otherwise},
\end{cases}
\end{equation}
where $L_1=\{l\neq k\mid a_{kl}\neq \bzero\}$, 
$L_2=\{l\neq m\mid a_{kl}\neq \bzero\}$ and $k\neq m$.
\begin{itemize}
\item If $a_{km}a_{mk}\leq \bunity$ then $\{x\mid A\otimes x\leq x\}$
is generated by the columns of $A^*$.
\item If $a_{km}a_{mk}>\bunity$ then $\{x\mid A\otimes x\leq x\}$
is generated by $e_i$ for $i\notin L_1\cup L_2\cup\{k\}\cup\{m\}$.
\end{itemize}
\end{proposition}
\begin{proof}
In the first case $A^*$ is finite and we apply Proposition 
\ref{klstar-gen}. For the second case observe that on one hand,
if $x_i\neq \bzero$ for some $i\in L_1\cup L_2\cup\{k\}\cup\{m\}$
then $x_k\neq \bzero$ and $x_m\neq \bzero$ which makes $A\otimes x\leq x$ 
impossible. On the other hand, any $x$ such that $x_i=\bzero$
for all $i\in L_1\cup L_2\cup\{k\}\cup\{m\}$ satisfies
$A\otimes x\leq x$.
\end{proof}

Now we proceed with writing out a generating set for solutions of
\eqref{e:sysmain}.
We denote $J_1:=\{i\mid a_{1i}\leq b_{1i},\ b_{1i}\neq \bzero\}$,
$J_2:=\{i\mid a_{2i}\leq b_{2i},\ b_{2i}\neq \bzero\}$, $I_1:=\{i\mid
a_{1i}>b_{1i}\}$ and $I_2:=\{i\mid a_{2i}>b_{2i}\}$. With this,
system \eqref{e:sysmain} is equivalent to
\begin{equation}
\label{e:sysred}
\begin{split}
\bigoplus_{i\in I_1} a_{1i}x_i& \leq \bigoplus_{i\in J_1} b_{1i} x_i,\\
\bigoplus_{i\in I_2} a_{2i}x_i& \leq \bigoplus_{i\in J_2} b_{2i} x_i.
\end{split}
\end{equation}
The solution set to \eqref{e:sysred} is the union of $S^{kl}$ defined by
\begin{equation}
\label{e:skldef}
S^{kl}=\{x\mid \bigoplus_{i\in I_1} a_{1i}x_i\leq b_{1k}x_k,\ \bigoplus_{i\in I_2} a_{2i}x_i\leq b_{2l}x_l\},
\end{equation}
for $k\in J_1$ and $l\in J_2$.
Further we represent $S^{kl}$ defined by \eqref{e:skldef}
in the form 
\begin{equation}
\label{skl-eig}
S^{kl}=\{x\mid A^{kl}\otimes x\leq x\},
\end{equation}
where we have to describe $A^{kl}$. There are two cases: $k=l$ and $k\neq l$. We denote $\gamma_{ki}^1:=b_{1k}^{-1}a_{1i}$ 
and $\gamma_{ki}^2:=b_{2k}^{-1}a_{2i}$. We also denote
$\oI_1:=\{1,\ldots,n\}\bez I_1$ and 
$\oI_2:=\{1,\ldots,n\}\bez I_2$. Observe
that $J_1\subseteq \oI_1$ and
$J_2\subseteq \oI_2$ (the containment may not be 
strict in general).

If $k=l$, then the $k$th row of $A^{kl}$ is 
\begin{equation}
\label{rows-akl1}
e'_k\oplus\bigoplus_{i\in I_1\cap\oI_2} \gamma_{ki}^1 e'_i\oplus
\bigoplus_{i\in\oI_1\cap I_2} \gamma_{ki}^2 e'_i\oplus\bigoplus_{i\in I_1\cap I_2} (\gamma_{ki}^1\oplus\gamma_{ki}^2)e'_i.
\end{equation}
and all other rows are row unit vectors.

If $k\neq l$ then the $k$th and the $l$th rows of $A^{kl}$ are 
given by
\begin{equation}
\label{rows-akl2}
e'_k\oplus\bigoplus_{i\in I_1} \gamma_{ki}^1 e'_i, \quad e'_l\oplus\bigoplus_{i\in I_2}\gamma_{li}^2 e'_i,
\end{equation}
all other rows being row unit vectors.

Now we collect the generators of $\{x\mid A^{kl}\otimes x\leq x\}$
considering several special cases.

{\em Case 1.} $k=l\in J_1\cap J_2$. 
The $k$th row of $A^{kl}$ is given by \eqref{rows-akl1} and
all other rows of $A^{kl}$ are unit vectors. By Proposition
\ref{verysimple}, $S^{kl}$ is generated by the
columns of $A^{kl}$. These are:
\begin{equation}
\label{cols-akl1}
\begin{split} 
& e_i,\  i\in\oI_1\cap\oI_2,\\
&\gamma_{ki}^1e_k\oplus e_i,\  k\in J_1\cap J_2,\ i\in I_1\cap\oI_2,\\ 
&\gamma_{li}^2e_l\oplus e_i,\  l\in J_1\cap J_2,\ i\in\oI_1\cap I_2,\\
&(\gamma_{ki}^1\oplus\gamma_{ki}^2)e_k\oplus e_i,\ k\in J_1\cap J_2,
\ i\in I_1\cap I_2.
\end{split}
\end{equation}

{\em Case 2.} $k\neq l$, $k\in J_1\cap\oI_2$, $l\in J_2\cap\oI_1$.\\
Rows $k$ and $l$ of $A^{kl}$ are given by \eqref{rows-akl2},
all other rows being the unit vectors.
As $l\in\oI_1$ and $k\in\oI_2$, we obtain
$A^{kl}_{kl}=\gamma_{kl}^1=\bzero$ and $A^{kl}_{lk}=\gamma_{lk}^2=\bzero$
and hence $(A^{kl})^*=A^{kl}$. 
Making transpose of \eqref{rows-akl2},
we obtain the columns of $(A^{kl})^*=A^{kl}$. 
By Proposition \ref{alsosimple} part 1. they generate  $S^{kl}$:
\begin{equation}
\label{cols-akl2}
\begin{split}
& e_i,\ i\in\oI_1\cap\oI_2,\\
& \gamma_{ki}^1e_k\oplus e_i,\ k\in J_1\cap\oI_2,\ i\in I_1\cap\oI_2,\\
& \gamma_{li}^2e_l\oplus e_i,\ l\in J_2\cap\oI_1,\ i\in\oI_1\cap I_2,\\ 
& \gamma_{ki}^1 e_k\oplus\gamma_{li}^2 e_l\oplus e_i,\ 
k\in J_1\cap\oI_2,\ l\in J_2\cap\oI_1,\ i\in I_1\cap I_2.
\end{split}
\end{equation}

{\em Case 3.} $k\in J_1\cap I_2$, $l\in J_2\cap\oI_1$.\\
Rows $k$ and $l$ of $A^{kl}$ are given by
\eqref{rows-akl2}. However, $(A^{kl})^*\neq A^{kl}$, since 
$k\in I_2$ implying that 
$A_{lk}^{kl}=\gamma_{lk}^2\neq \bzero$.
Note that $(A^{kl})^*$ is always finite, 
since $A_{kl}^{kl}=\bzero$ implying that the associated digraph of $A^{kl}$ does not contain any cycles with nonzero weight except
for the loops $(i,i)$. 
For $i\in I_1$,
we obtain 
$(A^{kl})^*_{li}=\gamma_{li}^2\oplus\gamma_{lk}^2\gamma_{ki}^1$. 
More precisely,
$(A^{kl})^*_{li}=\gamma_{lk}^2\gamma_{ki}^1$ for $i\in I_1\cap\oI_2$, and
$(A^{kl})^*_{li}=\gamma^2_{li}\oplus\gamma^2_{lk}\gamma_{ki}^1$ 
for $i\in I_1\cap I_2$.
The $l$th row of $(A^{kl})^*$ is given by
\begin{equation}
\label{rows-akl3}
%\begin{split}
e'_l\oplus\bigoplus_{i\in\oI_2\cap I_1} 
\gamma_{lk}^2\gamma_{ki}^1e'_i\oplus
\bigoplus_{i\in I_2\cap I_1} 
(\gamma_{li}^2\oplus\gamma_{lk}^2\gamma_{ki}^1)e'_i
\oplus\bigoplus_{i\in\oI_1\cap I_2}\gamma_{li}^2 e'_i.
%\end{split}
\end{equation}
The $k$th row of $(A^{kl})^*$ is the same as in \eqref{rows-akl2} and all other rows are unit vectors.
We obtain the columns of $(A^{kl})^*$:
\begin{equation}
\label{cols-akl3}
\begin{split}
& e_i,\  \text{for $i\in\oI_1\cap\oI_2$}\\
& e_i\oplus\gamma_{ki}^1e_k\oplus \gamma_{lk}^2\gamma_{ki}^1 e_l,\ 
k\in J_1\cap I_2,\ l\in J_2\cap\oI_1,\ i\in\oI_2\cap I_1\\
& e_i\oplus\gamma_{ki}^1e_k\oplus(\gamma_{li}^2\oplus\gamma_{lk}^2\gamma_{ki}^1)e_l,\ 
k\in J_1\cap I_2,\ l\in J_2\cap\oI_1,\ i\in I_1\cap I_2,\\ 
& e_i\oplus\gamma_{li}^2 e_l,\ l\in J_2\cap\oI_1,\ i\in\oI_1\cap I_2.
\end{split}
\end{equation}

{\em Case 4.} $k\in J_1\cap\oI_2$, $l\in J_2\cap I_1$.\\
Rows $k$ and $l$ are given by \eqref{rows-akl2}, and by analogy with Case 3 we obtain
that the $l$th row of $(A^{kl})^*$ is the same as in \eqref{rows-akl2}, but the $k$th row
is given by
\begin{equation}
\label{rows-akl4}
\begin{split}
e'_k \oplus\bigoplus_{i\in\oI_1\cap I_2} \gamma_{kl}^1\gamma_{li}^2e'_i\oplus
\bigoplus_{i\in I_2\cap I_1} (\gamma_{ki}^1\oplus\gamma_{kl}^1\gamma_{li}^2)e'_i
\oplus\bigoplus_{i\in\oI_2\cap I_1}\gamma_{ki}^1 e'_i.
\end{split}
\end{equation}
We obtain the columns of $(A^{kl})^*$:
\begin{equation}
\label{cols-akl4}
\begin{split}
& e_i,\ i\in\oI_1\cap\oI_2\\
& e_i\oplus\gamma_{li}^2e_l\oplus \gamma_{kl}^1\gamma_{li}^2 e_k,\ 
k\in J_1\cap\oI_2,\ l\in J_2\cap I_1,\ i\in\oI_1\cap I_2\\
& e_i\oplus\gamma_{li}^2e_l\oplus(\gamma_{ki}^1\oplus\gamma_{kl}^1\gamma_{li}^2)e_k,\ 
k\in J_1\cap\oI_2,\ l\in J_2\cap I_1,\ i\in I_1\cap I_2,\\ 
& e_i\oplus\gamma_{ki}^1 e_k,\
k\in J_1\cap\oI_2,\ i\in\oI_2\cap I_1.
\end{split}
\end{equation}

{\em Case 5.}  $k\in J_1\cap I_2$, $l\in J_2\cap I_1$.\\
If $\gamma_{lk}^2\gamma_{kl}^{1}\leq \bunity$, then
the $l$th row of $(A^{kl} )^*$ is given by \eqref{rows-akl3} 
and the $k$th row of $(A^{kl})^*$ is given by
\eqref{rows-akl4}. By Proposition \ref{alsosimple} part 1
the columns of $(A^{kl})^*$ generate $S^{kl}$. 
If $\gamma_{lk}^2\gamma_{kl}^1>\bunity$, then by Proposition
\ref{alsosimple} part 2, $S^{kl}$
is generated by $e_i$ for $i\in\oI_1\cap\oI_2$. 
If $\gamma_{lk}^2\gamma_{kl}^1\leq \bunity$, then
$(A^{kl})^*$ is finite and its columns are: 
\begin{equation}
\label{cols-akl5}
\begin{split}
& e_i,\ i\in\oI_1\cap\oI_2,\\
& e_l\oplus\gamma_{kl}^1e_k,\ e_k\oplus\gamma_{lk}^2e_l,\ 
k\in J_1\cap I_2,\ l\in J_2\cap I_1,\\
& e_i\oplus\gamma_{li}^2e_l\oplus \gamma_{kl}^1\gamma_{li}^2 e_k,\  
k\in J_1\cap I_2,\ l\in J_2\cap I_1,\ i\in\oI_1\cap I_2,\\
& e_i\oplus(\gamma_{li}^2\oplus\gamma_{lk}^2\gamma_{ki}^1)e_l\oplus(\gamma_{ki}^1\oplus\gamma_{kl}^1\gamma_{li}^2)e_k,\ 
k\in J_1\cap I_2,\ l\in J_2\cap I_1,\ i\in I_1\cap I_2,\\ 
& e_i\oplus\gamma_{ki}^1e_k\oplus\gamma_{lk}^2\gamma_{ki}^1 e_l,\ 
k\in J_1\cap I_2,\ l\in J_2\cap I_1,\ i\in\oI_2\cap I_1.
\end{split}
\end{equation}

\section{Identifying the basic solutions}
\label{s:extr}

A set $S\subseteq\Rpn$ is said to be {\em independent}
if no vector in this set is generated by other vectors in this
set. If such independent set generates a cone $K$ then it is called
a {\em basis} of $K$. It can be shown \cite{BSS-07,Wag-91} 
that if a basis of $K$ exists, then it consists of all
{\em extremals} (normalized in some sense): a vector
$x\in K$ is an extremal if $x=y\oplus z$ and $y,z\in K$ imply
$y=x$ or $z=x$. This also means that the basis of any cone is 
essentially unique: any two bases are obtained
from each other by multiplying their
elements by scalars. Importantly,
any finitely generated cone has a basis \cite{BSS-07,Wag-91}.

The notion of extremal defined above is a max-plus analogue
of the notion of extremal ray (or extremal) of a convex cone. It is also a special case of the join irreducible element of a lattice.

The extremality is most conveniently expressed by the following
multiorder principle \cite{BSS-07,GK-07,NS-07II,Ser-08} 
which we formulate
here only for finitely generated case. For any $i=1,\ldots,n$ we introduce the
relation
\begin{equation}
\label{leqi}
x\leq_i y\Rightarrow xx_i^{-1}\leq yy_i^{-1},\ \text{$x_i\neq \bzero$ 
and $y_i\neq \bzero$.}
\end{equation}
A vector $y\in K$ minimal with respect to $\leq_i$ will
be called {\em $i$-minimal}.

\begin{proposition}[Multiorder Principle]
Let $K\subseteq\Rpn$ be generated by a finite
set $S\subseteq\Rpn.$ Then $y\in S$ belongs to the basis
of $K$ (equivalently, is an extremal of $K$) 
if and only if it is $i$-minimal for some $i\in\{1,\ldots, n\}$.
\end{proposition}
\begin{proof}
If $y$ is not $i$-minimal for any $i$, then for each
$i\in\supp(y)$ there exists $z^i$ such that $z^i\leq_i y$.
Then it can be verified that
\begin{equation}
\label{e:maxcomb11}
y=\bigoplus_{i\in\supp(y)} z^i (z_i^i)^{-1} y_i.
\end{equation}
Conversely if $y=\bigoplus_k \alpha_k z^k$ for some $z^k\in S$,
then for each $i\in\supp(y)$ there is $k(i)$ such that
$y_i=\alpha_{k(i)}z_i^{k(i)}$ and as 
$y_j\geq \alpha_{k(i)} z_j^{k(i)}$ for all $j$ it follows
that $z^{k(i)}\leq_i y$ and $y$ is not $i$-minimal for any $i$.
\end{proof}
 
Next we classify all generators obtained in \eqref{cols-akl1},
\eqref{cols-akl2}, \eqref{cols-akl3},
\eqref{cols-akl4} and \eqref{cols-akl5} and give 
procedures for checking their extremality.
We start with unit vectors and combinations of two unit
vectors.

\vskip 1cm
$S_1$. $e_i,$ $\forall i\in\oI_1\cap\oI_2$.\\
$S_{2A1}$. $\phi_{ik}=\gamma_{ki}^1 e_k\oplus e_i,$ $\forall k\in J_1\cap\oI_2,$ $i\in I_1\cap\oI_2$.\\
$S_{2A2}$. $\phi_{ik}=\gamma_{ki}^2 e_k\oplus e_i,$ $\forall k\in J_2\cap\oI_1,$ $i\in I_2\cap\oI_1$.\\
$S_{2B}$. $\phi_{ik}=(\gamma_{ki}^1\oplus\gamma_{ki}^2)e_k\oplus e_i,$ $\forall k\in J_1\cap J_2,$ $i\in I_1\cap I_2$.\\
$S_{2C}$. $\phi_{lk}=\gamma_{kl}^1e_k\oplus e_l$ and $\phi_{kl}:=\gamma_{lk}^2e_l\oplus e_k,$ $\forall k\in J_1\cap I_2$, $l\in J_2\cap I_1$
such that $\gamma_{kl}^1\gamma_{lk}^2\leq \bunity$.\\

All vectors in $S_1$, $S_{2A}$ and $S_{2B}$ belong to the basis. 
Vectors in $S_{2C}$ belong to the basis
whenever they exist. For this, we determine the sets
\begin{equation}
\label{w-def}
\begin{split}
W&:=\{(k,l)\mid k\in J_1\cap I_2,\; l\in J_2\cap I_1,\;
\gamma_{kl}^1\gamma_{lk}^2\leq \bunity\}\\ \oW&:=\{(k,l)\mid k\in J_1\cap
I_2,\; l\in J_2\cap I_1,\; \gamma_{kl}^1\gamma_{lk}^2>\bunity\}
\end{split}
\end{equation}
Then, $\phi_{kl},\phi_{lk}\in S_{3C}$ exist whenever $(k,l)\in W$.
Note that if $\gamma_{kl}^1\gamma_{lk}^2=\bunity$ then $\phi_{kl}$ and
$\phi_{lk}$ are multiples of each other so that one of them can be
removed.

We proceed with combinations of three unit vectors.
Denote $K_1=\{i\mid a_{1i}=b_{1i}=\bzero\}$ and 
$K_2=\{i\mid a_{2i}=b_{2i}=\bzero\}$. Note that
$\{1,\ldots,n\}=I_1\cup J_1\cup K_1=I_2\cup J_2\cup K_2$.

\vskip 1cm
$S_{3A}$. $\psi_{ikl}=\gamma_{ki}^1e_k\oplus\gamma_{li}^2e_l\oplus e_i$ for $k\in J_1\cap\oI_2$,
$l\in J_2\cap\oI_1$, $i\in I_1\cap I_2$.\\

For all $i\in I_1\cap I_2$ determine the sets
\begin{equation}
\label{l-def}
\begin{split}
L_1(i)&:=\{k\in J_1\cap J_2\mid \gamma_{ki}^1<\gamma_{ki}^2\},\\
L_2(i)&:=\{l\in J_1\cap J_2\mid \gamma_{li}^2<\gamma_{li}^1\}.
\end{split}
\end{equation}
Then, $\psi_{ikl}\in S_{3A}$ belongs to the basis whenever
\begin{equation}
\label{3acond}
k\in(J_1\cap K_2)\cup L_1(i),\  l\in(J_2\cap K_1)\cup L_2(i).
\end{equation}

\vskip 1cm
\noindent $S_{3B1}$. $\psi_{ikl}=\gamma_{kl}^1\gamma_{li}^2e_k\oplus\gamma_{li}^2e_l\oplus e_i,$ $\forall k\in J_1\cap \oI_2$, $l\in J_2\cap I_1$,
$i\in I_2\cap \oI_1$.\\
$S_{3B2}$. $\psi_{ikl}=\gamma_{lk}^2\gamma_{ki}^1e_l\oplus\gamma_{ki}^1e_k\oplus e_i,$ $\forall k\in J_1\cap I_2$, $l\in J_2\cap\oI_1$,
$i\in\oI_2\cap I_1$.\\

For all $i\in I_2\cap\oI_1$, $l\in J_2\cap I_1$, determine the sets
\begin{equation}
\label{mil-def}
M_1(i,l):=\{t\in J_1\cap J_2\mid \gamma_{tl}^1\gamma_{li}^2<\gamma_{ti}^2\}.
\end{equation}
For all $i\in I_1\cap\oI_2$, $k\in J_1\cap I_2$, determine the sets
\begin{equation}
\label{mik-def}
M_2(i,k):=\{t\in J_1\cap J_2\mid \gamma_{tk}^2\gamma_{ki}^1<\gamma_{ti}^1\}.
\end{equation}
A vector in $\psi_{ikl}\in S_{3B1}$ 
(resp. $\psi_{ikl}\in S_{3B2}$) 
belongs to the basis if and only if
the following two conditions are satisfied:\\
1. $i\in I_2\cap K_1$ or $(i,l)\in\oW$
(resp. $i\in I_1\cap K_2$ or $(k,i)\in\oW$),\\
2. $k\in M_1(i,l)$ or $k\in J_1\cap K_2$ 
(resp. $l\in M_2(i,k)$ or $l\in J_2\cap K_1$).

\vskip 1cm

\noindent 
$S_{3C1}$. $\psi_{ikl}=(\gamma_{ki}^1\oplus\gamma_{kl}^1\gamma_{li}^2)e_k\oplus\gamma_{li}^2e_l\oplus e_i,$ $\forall k\in J_1\cap \oI_2$, $l\in J_2\cap I_1$,
$i\in I_2\cap I_1$.\\
$S_{3C2}$. $\psi_{ikl}=(\gamma_{li}^2\oplus\gamma_{lk}^2\gamma_{ki}^1)e_l\oplus\gamma_{ki}^1e_k\oplus e_i,$ $\forall k\in J_1\cap I_2$,
$l\in J_2\cap\oI_1$,
$i\in I_2\cap I_1$.\\

For all $i\in I_1\cap I_2$, $l\in J_2\cap I_1$, $k\in J_1\cap I_2$, determine the sets
\begin{equation}
\label{n-def}
\begin{split}
N_1(i,l):&=\{t\in J_1\cap J_2\mid \gamma_{ti}^1\oplus\gamma_{tl}^1\gamma_{li}^2<\gamma_{ti}^1\oplus\gamma_{ti}^2\}=\\
&=\{t\in L_1(i)\mid \gamma_{tl}^1\gamma_{li}^2<\gamma_{ti}^2\},\\
N_2(i,k):&=\{t\in J_1\cap J_2\mid \gamma_{ti}^2\oplus\gamma_{tk}^2\gamma_{ki}^1<\gamma_{ti}^1\oplus\gamma_{ti}^2\}=\\
&=\{t\in L_2(i)\mid \gamma_{tk}^2\gamma_{ki}^1<\gamma_{ti}^1\}.
\end{split}
\end{equation}
Then, $\psi_{ikl}\in S_{3C1}$ (resp. $\psi_{ikl}\in S_{3C2}$)
belongs to the basis if and only if
$k\in (J_1\cap K_2)\cup N_1(i,l)$
(resp. $l\in (J_2\cap K_1\cup N_2(i,k)$).

\vskip 1cm
$S_{3D1}$. $\psi_{ikl}=\gamma_{kl}^1\gamma_{li}^2e_k\oplus\gamma_{li}^2e_l\oplus e_i,$ $\forall k\in J_1\cap I_2$, $l\in J_2\cap I_1$,
$i\in I_2\cap \oI_1$ such that $\gamma_{kl}^1\gamma_{lk}^2\leq \bunity$.\\
$S_{3D2}$. $\psi_{ikl}=\gamma_{lk}^2\gamma_{ki}^1e_l\oplus\gamma_{ki}^1e_k\oplus e_i,$ $\forall k\in J_1\cap I_2$, $l\in J_2\cap I_1$,
$i\in\oI_2\cap I_1$ such that $\gamma_{kl}^1\gamma_{lk}^2\leq \bunity$.\\
$S_{3E}$. $\psi_{ikl}=(\gamma_{li}^2\oplus\gamma_{lk}^2\gamma_{ki}^1)e_l\oplus(\gamma_{ki}^1\oplus\gamma_{kl}^1\gamma_{li}^2)e_k\oplus e_i$,
$\forall k\in J_1\cap I_2$, $l\in J_2\cap I_1$, $i\in I_1\cap I_2$ such that $\gamma_{kl}^1\gamma_{lk}^2\leq \bunity$.\\

Provided that $(k,l)\in W$, vector $\psi_{ikl}\in S_{3D1}$ 
(resp. $\psi_{ikl}\in S_{3D2}$) belongs to the
basis if and only if $i\in K_1\cap I_2$ or 
$(i,l)\in\oW$ 
(resp. $i\in I_1\cap K_2$ or $(k,i)\in\oW$), and
$\psi_{ikl}\in S_{3E}$ always belong to the basis.

Below we explain why the above procedure yields the basis.
We denote by $S_1$ the set of all generators $e_i$ for $i\in\oI_1\cup\oI_2$,
by $S_2$ the set of all $2$-generators $\phi_{ik}$ and $\phi_{kl}$, and by $S_3$ the set
of all $3$-generators $\psi_{ikl}$.

$S_1,S_2$:\\  
The supports of all generators in $S_1\cup S_2$ are different, except for the
pairs of generators in $S_{2C}$, which exist if and only if $\gamma_{kl}^1\gamma_{lk}^2\leq \bunity$,
and are multiples of each other if and only if
$\gamma_{kl}^1\gamma_{lk}^2=\bunity$. Removing one vector from every such proportional
pair in $S_{2C}$ yields an independent set. Evidently, vectors in $S_1\cup S_2$ cannot be generated
with help of vectors in $S_3$, and this completes the explanation.

For the rest of the cases, first note that
the supports of all generators in $S_3$ are different and hence the set $S_3$ is independent.
It can only happen that the vectors in $S_3$ are linear combinations of the vectors in $S_1$ and $S_2$.

$S_{3A}$:\\ 
A vector $\psi_{ikl}\in S_{3A}$ may be a combination of vectors in $S_1$ and $S_{2B}$, as the
supports of some generators in these sets are contained in the support of a vector in $S_{3A}$.
By the minimality principle, a vector $\psi_{ikl}$ is extremal if and only if it is $i$-, $k$- or $l$-minimal.
Then, $\psi_{ikl}$ can be neither $k$- nor $l$-minimal since for all $k,l\in\oI_1\cap\oI_2$ the only
minimal generators are $e_k$ and $e_l$. The $i$-minimality of $\psi_{ikl}\in S_{3A}$ can be prevented only by
$\phi_{ki}\in S_{2B}$ or $\phi_{li}\in S_{2B}$. Condition \eqref{3acond} describes the situation when this
does not happen.

$S_{3B}$:\\
A vector $\psi_{ikl}\in S_{3B}$ can be a max combination of vectors in
$S_1$, $S_{2A}$ and $S_{2C}$ due to the inclusion of supports. Again, $\psi_{ikl}$ can
be neither $k$- nor $l$-minimal, since it can be represented as a combination
of $e_i$ and a vector from $S_{2A1}$ (resp. $S_{2A2}$) in the
case of $S_{3B1}$ (resp. $S_{3B2}$):
\begin{equation}
\begin{split}
\gamma_{kl}^1\gamma_{li}^2 e_k\oplus\gamma_{li}^2 e_l\oplus e_i&=\gamma_{li}^2(\gamma_{kl}^1e_k\oplus e_l)\oplus e_i.\\
\gamma_{lk}^2\gamma_{ki}^1 e_l\oplus\gamma_{ki}^1 e_k\oplus e_i&=\gamma_{ki}^1(\gamma_{lk}^2e_l\oplus e_k)\oplus e_i.
\end{split}
\end{equation}
Next we describe the 2-generators which can prevent the $i$-minimality of $\psi_{ikl}\in S_{3B1}$ (resp.
$\psi_{ikl}\in S_{3B2}$).\\ 
1. $\phi_{il},\phi_{li}\in S_{2C}$ (resp.  $\phi_{ki},\phi_{ik}\in
S_{2C}$). These 2-generators do not arise
only if $i\in K_1$ for $S_{3B1}$ (resp. $i\in K_2$ for
$S_{3B2}$), for then there is no vector in $S_{2C}$ whose support is
a subset of the support of $\psi_{ikl}$, or if the corresponding
pair $\phi_{il},\phi_{li}\in S_{2C}$ (resp. $\phi_{ki},\phi_{ik}\in
S_{2C}$) does not exist meaning $(i,l)\in\oW$ (resp. $(k,i)\in\oW$).\\ 
2. $\phi_{ik}\in S_{2A2}$ (resp. $\phi_{il}\in S_{2A1}$).
These vectors do not arise only if $k\in K_2$ (resp. $l\in K_1$),
because then $k\notin J_2$ (resp. $l\notin J_1$) unlike in
the case of $S_{2A2}$ (resp. $S_{2A1}$). Otherwise, $\phi_{ik}$
(resp. $\phi_{il}$) are not dangerous, i.e., they do not precede
$\psi_{ikl}$ with respect to $\leq_i$ only if $k\in M_1(i,l)$
(resp. $l\in M_2(i,k)$), see \eqref{mil-def} and \eqref{mil-def}.

$S_{3C}$:\\
A vector $\psi_{ikl}\in S_{3C}$ can be a max combination of vectors in
$S_1$, $S_{2A}$ and $S_{2B}$. Again, $\psi_{ikl}$ can be neither $k$- nor $l$- minimal.
Indeed,
\begin{equation}
\label{case4-lindep}
\begin{split}
\psi_{ikl}&=\gamma_{ki}^1e_k\oplus e_i\oplus\gamma_{li}^2(\gamma_{kl}^1e_k\oplus e_l),\ S_{3C1}\\
\psi_{ikl}&=\gamma_{li}^2e_l\oplus e_i\oplus\gamma_{ki}^1(\gamma_{lk}^2e_l\oplus e_k),\ S_{3C2}
\end{split}
\end{equation}
where the vectors in brackets belong to $S_{2A1}$ and $S_{2A2}$ respectively. The first vector
cannot be $k$-minimal since $k\in\oI_1\cap\oI_2$, and it cannot be $l$-minimal as it to loses
$\gamma_{kl}^2e_k\oplus e_l\in S_{2A1}$.
The second vector cannot be $l$-minimal since $l\in\oI_1\cap\oI_2$, and it cannot be
$k$-minimal as it loses to $\gamma_{lk}^2e_l\oplus e_k\in S_{2A2}$.  The remaining possibility of being
$i$-minimal can be destroyed by vectors from $S_{2B}$, and this does not happen if and only if
the given conditions are satisfied.

$S_{3D}$, $S_{3E}$:\\
A vector $\psi_{ikl}\in S_{3D}$ cannot be a
max combination of other vectors of $S_2$ than those
in $S_{2C}$. It is not a max combination of vectors in
$S_{2C}$ only if $i$ is not suitable for existence
of vectors in $S_{2C}$. This happens if  
$i\in K_1\cap I_2$ or $(i,l)\in\oW$ for the case $\psi_{ikl}\in S_{3D1}$,
and $i\in I_1\cap K_2$ or $(k,i)\in\oW$ for the case $\psi_{ikl}\in S_{3D2}$.
Finally, the vectors in $S_{3E}$ cannot be combinations of vectors in $S_2$, since only vectors in $S_{2C}$ have relevant supports (and yet not enough). So the vectors in $S_{3E}$ are in the basis whenever they exist.

We note that the complexity of the above procedure id $O(n^3)$,
which is due to the computation of the sets $M_1(i,l)$ 
\eqref{mil-def}, $M_2(i,k)$ \eqref{mik-def}, $N_1(i,l)$ and
$N_2(i,k)$ \eqref{n-def}, and checking conditions for all
combinations of three unit vectors.

We conclude the paper with two examples. The second example is taken from \cite{WTFT}, Example 4.2.

{\em Example 1.} To illustrate the sets of generators
constructed in the paper on a simple example, we consider
the following system of two inequalities with four variables:
\begin{equation}
\label{e:sysex}
\begin{split}
& 4\otimes x_3\oplus 2\otimes x_4 \leq x_1\oplus 2\otimes x_2,\\
& 3\otimes x_1\oplus x_3 \leq x_2.
\end{split}
\end{equation}
We have $I_1=\{3,4\}$, $J_1=\oI_1=\{1,2\}$,
$I_2=\{1,3\}$, $J_2=\{2\}$, $\oI_2=\{2,4\}$.
We compute\\
$S_1$: just $e_2$, since $\oI_1\cap\oI_2=\{2\}$;\\
$S_{2A1}$: just $\gamma_{24}^1e_2\oplus e_4=e_2\oplus e_4$,
since $J_1\cap\oI_2=\{2\}$ and $I_1\cap\oI_2=\{4\}$;\\
$S_{2A2}$: just $\gamma_{21}^2e_2\oplus e_1=3e_2\oplus e_1$,
since $J_2\cap\oI_1=\{2\}$ and $I_2\cap\oI_1=\{4\}$;\\
$S_{2B}$: $(\gamma_{23}^1\oplus\gamma_{23}^2)e_2\oplus e_3=2e_2\oplus e_3$,
since $J_1\cap J_2=\{2\}$ and $I_1\cap I_2=\{3\}$;\\
$S_{2C}$: empty, since $J_2\cap I_1$ is empty;\\
$S_{3A}$: trivializes to $S_{2B}$;\\ 
$S_{3B1}$: empty, since $J_2\cap I_1$ is empty;\\
$S_{3B2}$: just $\gamma_{21}^2\gamma_{14}^1 e_2\oplus \gamma_{14}^1 e_1\oplus e_4=5e_2\oplus 2e_1\oplus e_4$,
since $J_1\cap I_2=\{1\}$, $J_2\cap \oI_1=\{2\}$, 
$\oI_2\cap I_1=\{4\}$;\\
$S_{3C1}$: empty, since $J_2\cap I_1$ is empty;\\
$S_{3C2}$: just $(\gamma_{23}^2\oplus \gamma_{21}^2\gamma_{13}^1) e_2\oplus \gamma_{13}^1 e_1\oplus e_3$, which is $7e_2\oplus 4e_1\oplus e_3$,
since $J_1\cap I_2=\{1\}$, $J_2\cap \oI_1=\{2\}$, 
$I_2\cap I_1=\{3\}$;\\
$S_{3D1}$, $S_{3D2}$ and $S_{3E}$: empty, since $J_2\cap I_1$ is empty.

In this example, the basis consists of four generators in
$S_1$, $S_{2A1}$, $S_{2A2}$ and $S_{2B}$: $e_2$, 
$e_2\oplus e_4$, $3e_2\oplus e_1$ and $2e_2\oplus e_3$.
Indeed, the remaining two generators in $S_3$ are redundant:
1) $5e_2\oplus 2e_1\oplus e_4$ ($S_{3B2}$) is a combination
of $e_2\oplus e_4$ ($S_{2A1}$) and $3e_2\oplus e_1$ ($S_{2A2}$),
2) $7e_2\oplus 4e_1\oplus e_3$ ($S_{3C2}$) is a combination
of $3e_2\oplus e_1$ ($S_{2A2}$) and $2e_2\oplus e_3$ ($S_{2B}$).

{\em Example 2.}
To compare our results with the approach of \cite{WTFT},
we consider \cite{WTFT}, Example 4.2, which is a system
of two inequalities with seven variables:
\begin{equation}
\label{e:sysex2}
\begin{split}
& x_4\oplus 4\otimes x_5\oplus 2\otimes x_6\oplus 6\otimes x_7\leq x_1\oplus 1\otimes x_2\oplus 5\otimes x_3,\\
& 5\otimes x_2\oplus 6\otimes x_3\oplus 2\otimes x_7\leq
3\otimes x_1\oplus x_4\oplus 2\otimes x_5\oplus 4\otimes x_6.
\end{split}
\end{equation} 
In this case $I_1=\{4,5,6,7\}$, $J_1=\{1,2,3\}=\oI_1$,
$I_2=\{2,3,7\}$, $J_2=\{1,4,5,6\}=\oI_2$.  We compute the
generators comparing them with those in the table of
\cite{WTFT} page 365:\\
$S_1:$ just $e_1$, since $\oI_1\cap\oI_2=\{1\}$. This is $x_1$ in the table of \cite{WTFT}.\\
$S_{2A1}:$ Combining $J_1\cap\oI_2=\{1\}$ and
$I_1\cap\oI_2=\{4,5,6\}$ we obtain $e_1\oplus e_4$, $4e_1\oplus e_5$ and $2e_1\oplus e_6$. Vector $e_1\oplus e_4$ corresponds to $x_3$, and the remaining two vectors are $x_5$ and $x_{10}$ in the table of \cite{WTFT}.\\
$S_{2A2}:$ Combining $J_2\cap \oI_1=\{1\}$ and
$I_2\cap \oI_1=\{2,3\}$ we obtain $2e_1\oplus e_2$ and $3e_1\oplus e_3$. These correspond to $x_4$ and $x_7$ in the table of \cite{WTFT}.\\
$S_{2B}:$ just $6e_1\oplus e_7$, combining
$J_1\cap J_2=\{1\}$ with $I_1\cap I_2=\{7\}$. This is
$x_2$ in the table of \cite{WTFT}.\\
$S_{2C}.$ To compute these we need to combine
$J_1\cap I_2=\{2,3\}$ with $J_2\cap I_1=\{4,5,6\}$. For each
$k=2,3$ and $l=4,5,6$ we need to check whether $\gamma_{kl}^1\gamma_{lk}^2\leq \bunity$, and each time
this condition is satisfied we have two vectors (or just
one vector if $\gamma_{kl}^1\gamma_{lk}^2=\bunity$). In our case
the condition is satisfied only with $k=3$ and $l=6$. This yields
two vectors $2e_6\oplus e_3$ and $e_3\oplus 3e_6$, which are
$x_6$ and $x_{11}$ in the table of \cite{WTFT}.\\
$S_{3A}:$ trivializes to $S_{2B}$.\\
$S_{3B1}:$ We need to combine $J_1\cap \oI_2=\{1\}$,
$J_2\cap I_1=\{4,5,6\}$ and $I_2\cap\oI_1=\{2,3\}$.
For $i=2,3$, $l=4,5,6$ and $k=1$, each time when 
$\gamma_{il}^1\gamma_{li}^2>\bunity$, we have to verify
whether $\gamma_{kl}^1\gamma_{li}^2<\gamma_{ki}^2$ holds.
Each time when both conditions are satisfied, we have an independent vector of the basis. Here it never happens.\\
$S_{3B2}:$ We need to combine $J_1\cap I_2=\{2,3\}$,
$J_2\cap \oI_1=\{1\}$ and $\oI_2\cap I_1=\{4,5,6\}$.
For $k=2,3$, $i=4,5,6$ and $l=1$, each time when 
$\gamma_{ki}^1\gamma_{ik}^2>\bunity$, we have to verify
whether $\gamma_{lk}^2\gamma_{ki}^1<\gamma_{li}^1$ holds.
Each time when both conditions are satisfied, we have an independent vector of the basis. Here it happens with
1) $l=1$, $k=3$ and $i=4$ leading to $3e_1\oplus e_3\oplus 5e_4$
which corresponds to $x_8$ of \cite{WTFT}, 2)
$l=1$, $k=3$ and $i=5$ leading to $3e_1\oplus e_3\oplus 1e_5$,
which corresponds to $x_9$ of \cite{WTFT}. \\
$S_{3C1}:$ Here we combine $J_1\cap\oI_2=\{1\}$ with
$J_2\cap I_1=\{4,5,6\}$ and $I_1\cap I_2=\{7\}$. Since $\gamma_{ki}^1>\gamma_{ki}^2$ with $k=1$ and $i=7$, no vector belongs to the basis
in this case.\\
$S_{3C2}:$ We combine $J_1\cap I_2=\{2,3\}$, $J_2\cap\oI_1=\{1\}$
and $I_1\cap I_2=\{7\}$. For each $k=2,3$, $l=1$ and $i=7$
we have to verify $\gamma_{li}^2\oplus \gamma_{lk}^2\gamma_{ki}^1<\gamma_{li}^2\oplus \gamma_{li}^1$.
This happens for $k=3$, $l=1$ and $i=7$ and yields the vector
$4e_1\oplus 1e_3\oplus e_7$, which corresponds to $x_{12}$
of \cite{WTFT}.\\
$S_{3D1}:$ We combine $J_1\cap I_2=\{2,3\},$ 
$J_2\cap I_1=\{4,5,6\}$, $I_2\cap\oI_1=\{2,3\}$.
For $k=2,3$ and $l=4,5,6$, the condition 
$\gamma_{lk}^1\gamma_{kl}^2\leq\bunity$ holds only for 
$k=3$ and $l=6$, so it remains to verify $\gamma_{il}^1\gamma_{li}^2>\bunity$ for $i=2$ and $l=6$.  
This condition holds and we obtain $\gamma_{36}^1\gamma_{62}^2e_3
\oplus\gamma_{62}^2 e_6\oplus e_2$ which is proportional with
$2e_2\oplus e_3\oplus 3e_6$. 
Note that the max-linear combination of $e_2,e_3,e_6$  given for  $x_{13}$ in the table of \cite{WTFT} is an error, since $Ax_{13} \not\le Bx_{13}$. \\
$S_{3D2}:$ We combine $J_1\cap I_2=\{2,3\},$ 
$J_2\cap I_1=\{4,5,6\}$, $\oI_2\cap I_1=\{4,5,6\}$.
For $k=2,3$ and $l=4,5,6$, the condition 
$\gamma_{lk}^1\gamma_{kl}^2\leq\bunity$ holds only for 
$k=3$ and $l=6$, so it remains to verify $\gamma_{ki}^1\gamma_{ik}^2>\bunity$ for $i=4,5$ and $k=3$.  
This condition holds in both cases and yields $\gamma_{63}^2\gamma_{34}^1e_6
\oplus\gamma_{34}^1 e_3\oplus e_4$ which is proportional with
$2e_6\oplus e_3\oplus 5e_4$, and $\gamma_{63}^2\gamma_{35}^1e_6
\oplus\gamma_{35}^1e_3\oplus e_5$ proportional with
$2e_6\oplus e_3\oplus 1e_5$.\\
$S_{3E}:$ We combine $J_1\cap I_2=\{2,3\},$ $J_2\cap I_1=\{4,5,6\}$ and $I_1\cap I_2=\{7\}$. As
$\gamma_{kl}^1\gamma_{lk}^2<\bunity$ only for $k=3$ and $l=6$,
we have only one generator, namely $3e_6\oplus 1e_3\oplus e_7$.
\vskip 1cm 

Thus the basis consists of $e_1,$ $8$ combinations of $2$ unit vectors and $7$ combinations of $3$ unit vectors.

The $2$-combinations are: $e_1\oplus e_4$, $4e_1\oplus e_5$ and $2e_1\oplus e_6$ ($S_{2A1}$), $2e_1\oplus e_2$ and $3e_1\oplus e_3$ ($S_{2A2}$),
$6e_1\oplus e_7$ ($S_{2B}$), $2e_6\oplus e_3$ and 
$e_3\oplus 3e_6$ ($S_{2C}$).

The $3$-combinations are: $3e_1\oplus e_3\oplus 5e_4$,
$3e_1\oplus e_3\oplus 1e_5$ ($S_{3B2}$), $4e_1\oplus 1e_3\oplus e_7$ ($S_{3C2}$), $2e_2\oplus e_3\oplus 3e_6$ ($S_{3D1}$),
$2e_6\oplus e_3\oplus 5e_4$ and $2e_6\oplus e_3\oplus 1e_5$
($S_{3D2}$), $3e_6\oplus 1e_3\oplus e_7$ ($S_{3E}$).

We note that all vectors that we have found, are solutions of the
system, and moreover, all $3$-generators turn both inequalities into equalities, which in analogy with the convex analysis also suggests that they must be extremals (the $2$-generators correspond to the intersections with coordinate planes). Actually vectors in $S_{3B2}$ and $S_{3C2}$ are different from
$x_8,x_9$ and $x_{12}$ from the table of \cite{WTFT} page 365,
to which they correspond in terms of supports. For these, $x_8=4e_1\oplus e_3\oplus 4e_4$
is a combination of $3e_1\oplus e_3\oplus 5e_4$ (from $S_{3B2}$),
$3e_1\oplus e_3$ (from $S_{2A2})$ and $e_1$, 
$x_9=4e_1\oplus 1e_3\oplus e_5$ is a combination of $3e_1\oplus e_3\oplus 1e_5$ (from $S_{3B2}$)
and $3e_1\oplus e_3$, and $x_{12}=5e_1\oplus 1e_3\oplus e_7$ is a
combination of $4e_1\oplus 1e_3\oplus e_7$ (from $S_{3C2}$)
and $e_1$. The remaining generator in the table of
\cite{WTFT} is $x_{13}=e_2\oplus 2e_3\oplus 1e_6$. It is in error, since it violates the second inequality of \eqref{e:sysex2}, but in terms of support, it corresponds to
$2e_2\oplus e_3\oplus 3e_6$ from $S_{3D1}$. Also, there are
three combinations which are not in the table of \cite{WTFT},
from $S_{3D2}$ and $S_{3E}$.

%\bibliographystyle{plain}
%\bibliography{twoineqs-basis}

\end{document}